\documentclass{amsart}

\usepackage{geometry}
\usepackage{amsmath}
\usepackage{amssymb}
\usepackage{amsthm}
\usepackage{mathtools}
\usepackage{tikz}

\usepackage{float}

\theoremstyle{definition}
\newtheorem{thm}{Theorem}[section]
\newtheorem{cor}[thm]{Corollary}

\newtheorem{lem}[thm]{Lemma}

\newtheorem{defn}[thm]{Definition}

\newtheorem{rmk}[thm]{Remark}

\newtheorem*{ack*}{Acknowledgements}


\begin{document}

\title{Orthogonal Symmetric Chain Decompositions of Hypercubes}
\author{Hunter Spink}
\address{Harvard University, Cambridge, 1 Oxford Street, 02138}
\email{hspink@math.harvard.edu}
\thanks{The author would like to thank Harvard University for providing travel funding.}

\begin{abstract}
In 1979, Shearer and Kleitman \cite{SK} conjectured that there exist $\lfloor n/2 \rfloor+1$ orthogonal chain decompositions of the hypercube $Q_n$, and constructed two orthogonal chain decompositions. In this paper, we make the first non-trivial progress on this conjecture since \cite{SK} by constructing three orthogonal chain decompositions of $Q_n$ for $n$ large enough. To do this, we introduce the notion of ``almost orthogonal symmetric chain decompositions''. We explicitly describe three such decompositions of $Q_5$ and $Q_7$, and describe conditions which allow us to decompose products of hypercubes into $k$ almost orthogonal symmetric chain decompositions given such decompositions of the original hypercubes.
\end{abstract}

\maketitle

\section{Introduction}
Symmetric chain decompositions of posets as an object of study were perhaps first introduced by Kleitman in his groundbreaking 1970's paper \cite{LO} resolving the Littlewood-Offord problem on sums of Bernoulli random variables in an arbitrary number of dimensions. The inductive decomposition of the hypercube $Q_n=\{0,1\}^n$ (which is alternatively described as the power set of $[n]=\{1,2,\ldots, n\}$ via indicator functions) produced by the ``duplication'' technique introduced in that paper was later streamlined in Greene-Kleitman \cite{GK}, where a clever parenthesis-matching argument replaced the original inductive decomposition. It is in this form that many authors have attempted to exploit the method. Recently this was carried out by Hersh-Schilling \cite{lyndon}, where they produced a symmetric chain decomposition of the quotient of $Q_n$ by the natural $\mathbb{Z}/n\mathbb{Z}$-action using Lyndon words on the parentheses. However, one of the first applications of the Greene-Kleitman technique was in 1979, when Shearer and Kleitman \cite{SK} used the technique to produce two of what they called ``orthogonal chain decompositions'' of the hypercube $Q_n=\{0,1\}^n$, in order to find a lower bound on the probability that two randomly chosen subsets of $\{1,2,\ldots, n\}$ are comparable. Here two decompositions $\mathcal{F}$ and $\mathcal{G}$ of $Q_n$ into $n \choose {n/2}$ chains (which is the minimal number possible by Sperner's Theorem) are \textit{orthogonal chain decompositions} if for any chain $A$ in $\mathcal{F}$ and any chain $B$ in $\mathcal{G}$, we have $|A \cap B| \le 1$.

The inductive decomposition presented in \cite{LO} can be viewed as a very special case of the following general observation on products of posets with symmetric chain decompositions. Suppose that $P$ and $Q$ are two finite graded posets which have symmetric chain decompositions, then $P \times Q$ has a symmetric chain decomposition which is formed by taking every product of a chain in the decomposition of $P$ with a chain in the decomposition of $Q$, and decomposing the resulting rectangles into symmetric chains. This observation is crucial to the present paper. When we take the special case of $P=Q_k$ and $Q=Q_1$, we obtain precisely the duplication method of Kleitman. We shall see in Section 4 how the argument in \cite{SK} based off of Greene-Kleitman to produce two orthogonal chain decompositions is in some sense a simple consequence of the fact that a $2 \times m$ rectangle has two possible symmetric chain decompositions.

In \cite{SK}, it was conjectured that $\lfloor n/2 \rfloor+1$ orthogonal chain decompositions exist on $Q_n$. Since their paper in 1979, only two orthogonal chain decompositions were known to exist in general. The present paper makes the first non-trivial progress on this conjecture. An immediate corollary of our first main result Theorem 3.1 is the construction of three orthogonal chain decompositions of $Q_n$ for $n$ large enough. As in \cite{SK}, the present paper proceeds by modifying symmetric chain decompositions which almost satisfy the above orthogonality condition between them. We make this explicit in the present paper by saying that two symmetric chain decompositions $\mathcal{F}$ and $\mathcal{G}$ of $Q_n$ are $\textit{almost orthogonal symmetric chain decompositions}$ if for any chain $A$ in $\mathcal{F}$ and any chain $B$ in $\mathcal{G}$, we have $|A \cap B|\le 1$ with the exception that if $A$ and $B$ are both maximal chains (note that there is a unique maximal chain in any symmetric chain decomposition of $Q_n$), then we require that $A$ and $B$ intersect in precisely their maximal and minimal elements.

If we have two almost orthogonal symmetric chain decompositions of $Q_n$ for $n \ge 2$, then it is easy to see that we can move around the empty set in one of the decompositions to produce orthogonal chain decompositions, which is precisely how Shearer and Kleitman produced their pair of orthogonal chain decompositions. We can similarly move around the empty set if we have $k \ge 3$ almost orthogonal symmetric chain decompositions. Indeed, by considering the elements of middle rank(s), one can first show that $k \le \lfloor n/2 \rfloor +1$. Hence when $n \ge 5$ we have the number of minimal length chains in a symmetric chain decomposition ${n \choose \lfloor n/2\rfloor }-{n \choose \lfloor n/2\rfloor-1}$ is strictly larger than $k-1$ if $n$ is even and $2(k-1)$ if $n$ is odd, so we can move the empty set to a chain of minimal length one decomposition at a time without ever creating bad intersections. For $n=4$, one can either employ a more careful analysis of the $1$-element chains, or show by tedious casework that there does not exist three almost orthogonal symmetric chain decompositions (which we will not discuss further in this paper). This allows us to focus our attention on producing collections of $k \ge 3$ almost orthogonal symmetric chain decompositions of $Q_n$.

Similar to how Shearer and Kleitman's proof can be interpreted as using the two possible symmetric chain decompositions on $2 \times m$ to produce orthogonal chains in the product of $Q_1$ by $Q_{n-1}$, the strategy carried out in this paper will be as follows. We take $k$ almost orthogonal decompositions $\mathcal{F}_j^i$ on $Q_{n_i}$ with $n_i \ge 2$ for each $i=1,2,\ldots, r$, and consider cuboids formed for each $j=1,2,\ldots, k$  by the products of symmetric chains, one from each $\mathcal{F}^j_i$ for $i=1,2,\ldots, r$. For each fixed $j$, we decompose the resulting cuboids, and obtain a symmetric chain decomposition of $Q_{n_1}\times \ldots\times Q_{n_r}=Q_{n_1+\ldots +n_r}$. Decomposing the cuboids in such a way that we preserve the pairwise almost orthogonality between the decompositions for various $j$ proves to be a daunting challenge, and cannot always be done without additional conditions. We comment on these conditions below.

First, one can show that if $Q_{n_i}$ is even dimensional, then we have no hope of decomposing the product almost orthogonally unless all of the $1$-element chains in the decompositions of $Q_{n_i}$ are disjoint (by e.g. considering what happens if we take maximal chains in the remaining factors). For $n_i$ odd, we will see in Section 5 that the analogous natural condition to impose on $Q_{n_i}$ to make products of two hypercubes work is the following. If we take the natural bipartite graph formed by the two-element chains in all decompositions of $Q_{n_i}$, then we require the existence of an orientation on its edges such that every vertex has out-degree at most $1$. In the case of $k=2$, this always holds as the union of two partial matchings on a bipartite graph is the union of cycles and paths. We define a collection of almost orthogonal symmetric chain decompositions to be $\textit{good}$ if it satisfies the corresponding condition above (depending on whether $n_i$ is even or odd). As it turns out, we will show in Theorem 5.6 that in the case of two hypercubes, being able to decompose their product into $k$ almost orthogonal decompositions is equivalent to requiring that both hypercubes are good. The second main result Theorem 3.3 says that in the product of hypercubes, we can guarantee a decomposition of the product by requiring the goodness hypothesis on the even dimensional hypercubes and two of the odd dimensional hypercubes (or as many odd dimensional hypercubes as are present in the product if there are fewer than two).

In Subsections 4.2 and 4.3, we explicitly describe collections of three almost orthogonal chain decompositions of $Q_5$ and of $Q_7$, each collection satisfying the above goodness condition. Together with Theorem 3.3, we can conclude Theorem 3.1.

A second, more surprising condition we could impose to guarantee that the product has $k$ almost orthogonal symmetric chain decompositions is seen in the third main result Theorem 3.4. Here we still require the goodness of the even dimensional hypercubes, but for the odd dimensional hypercubes we only impose that there are at least six of them. The proof crucially relies on some very specific symmetric chain decompositions provided in \cite{taut} of cuboids which are formed by the product of at least five $2$-element chains and a chain of length at least $5$.

To aid us in decomposing cuboids coherently, we introduce in Definition 4.1 ``proper'' and ``very proper'' decompositions. These notions guarantee the necessary intersection relations between the chains in the ``(very) proper'' decomposition of the cuboid and symmetric chains in any decomposition of any other cuboid under consideration. This allows us to avoid analyzing the intricate relations between most of the symmetric chain decompositions of the various cuboids, and reduces us to only considering the interplays between the symmetric chain decompositions in a small number of cases. As it turns out, cuboids formed by products of minimal length chains and maximal chains are mostly responsible for the number of cases we will have to handle. Our approach will be systematic, with the difficulty of the arguments decreasing as our toolbox increases in size.

If one only desires to know the proof of Corollary 3.2 (resp. Theorem 3.1) on the existence of three (almost) orthogonal (symmetric) chain decompositions for $n$ large enough, note it is a consequence of Sections 4.2, 4.3, repeated applications of Theorem 6.1, (Theorem 7.1,) and the argument in Corollary 10.7.

The structure of this paper is as follows.

\begin{itemize}
\item In Section 2, we give definitions.
\item In Section 3, we state our main results.
\item In Section 4, we lay the groundwork for the rest of the paper. We first discuss the proof of the two orthogonal decomposition result and its limitations to generalization. Then we produce three almost orthogonal symmetric chain decompositions of $Q_5$ and $Q_7$, and give the notion of a (very) proper decomposition into symmetric chains, which allows us to guarantee orthogonality without having to consider the complex interplay between different decompositions.
\item In Section 5, we attempt to decompose a product of almost orthogonal symmetric chain decompositions. Chains of very small length turn out to pose a non-trivial technical obstruction to making this strategy work, and we encode the precise conditions we need to make everything go through in the definition of goodness. The following four sections mitigate that the product of two good decompositions is not necessarily good.
\item In Section 6, we consider the product of three hypercubes, with two having good decompositions, and one being either odd dimensional or having a good decomposition (an even hypercube without a good decomposition will always fail to work in a product construction).
\item In Section 7, we consider the product of four odd dimensional hypercubes, with two having good decompositions.
\item In Section 8, we consider the product of arbitrarily many good even hypercubes (which is necessary for Theorem 3.3, but not necessary for Theorem 3.1 or Corollary 3.2).
\item In Section 9, we consider the product of 1 odd dimensional hypercube and 3 good even dimensional hypercubes (again, not necessary for Theorem 3.1 or Corollary 3.2).
\item In Section 10, we prove Theorems 3.1 and 3.3, using the results from the previous sections.
\item In Section 11, we prove Theorem 3.4.
\end{itemize}

\section{Definitions}

Following Shearer and Kleitman \cite{SK}, call two decompositions $\mathcal{F}$ and $\mathcal{G}$ of $Q_n$ into $n \choose \lfloor n/2 \rfloor$ chains (which is the minimal possible by Sperner's Theorem) \textit{orthogonal} if for any chain $A$ in $\mathcal{F}$ and any chain $B$ in $\mathcal{G}$, we have $|A \cap B| \le 1$. Say a family of decompositions of $Q_n$ into $n \choose \lfloor n/2 \rfloor$ chains is \textit{orthogonal} if every pair is.

Recall a \textit{symmetric chain decomposition} of $Q_n$ is a decomposition of $Q_n$ into \textit{symmetric chains}, i.e. for each chain there is an integer $k$ such that the chain consists of one element of size $k,k+1,\ldots, n-k$; this is always a decomposition into $n \choose \lfloor n/2 \rfloor$ chains since each chain contains a $\lfloor n/2 \rfloor$-element subset. The following definition is at the heart of this paper.

\begin{defn}
We define symmetric chain decompositions $\mathcal{F}$ and $\mathcal{G}$ of $Q_n$ to be \textit{almost orthogonal} if for $A$ a chain in $\mathcal{F}$ and $B$ a chain in $\mathcal{G}$, $|A \cap B| \le 1$ except when $A$ and $B$ both contain $\varnothing$ and $[n]$ (a unique such chain exists in any symmetric chain decomposition), in which case we require instead $|A \cap B|=2$. Say a family of symmetric chain decompositions is \textit{almost orthogonal} if every pair is.
\end{defn}

To formulate Theorem 3.3, we restate from the introduction the following technical condition, which will be fully motivated in Section 5.

\begin{defn}
We say that a collection of almost orthogonal symmetric chain decompositions $\mathcal{F}_i$ on $Q_n$ is $\textit{good}$ if either $n$ is even and the $1$-element chains in the $\mathcal{F}_i$ are all distinct, or $n$ is odd and the graph with vertex set the union of all $2$-element chains in all $\mathcal{F}_i$ and edges the $2$-element chains themselves can have its edges oriented so that every vertex has out-degree at most $1$ (equivalently, every component is a tree or a tree union an edge).
\end{defn}

\section{Main Results}
We now state our main results. Theorems 3.1 and 3.3, as well as Corollary 3.2, are proved in Section 10, while Theorem 3.4 is proved in Section 11.

\begin{thm}
There exist 3 almost orthogonal symmetric chain decompositions of $Q_n$ for all $n$ large enough.
\end{thm}

By moving around $\varnothing$ between the chains (see the introduction), this readily implies the following.

\begin{cor}
There exist 3 orthogonal decompositions of $Q_n$ for all $n$ large enough.
\end{cor}

This makes progress on the conjecture in Shearer and Kleitman's paper \cite{SK} (which was verified up to $n=4$) that there exist $\lfloor n/2 \rfloor +1$ orthogonal chain decompositions of $Q_n$, easily shown to be an upper bound by considering elements of middle rank(s). We can say the following for $k \ge 3$ almost orthogonal symmetric chain decompositions.

\begin{thm}
For $1 \le i \le r$, suppose we have $k \ge 3$ almost orthogonal symmetric chain decompositions $\mathcal{F}^j_i$ of $Q_{n_i}$, $n_i \ge 2$. Further suppose that $\{\mathcal{F}^j_i\}$ is a good collection whenever $n_i$ is even, and either all or at least two of the $\{\mathcal{F}^j_i\}$ are good collections for $n_i$ odd. Then we can construct $k$ almost orthogonal symmetric chain decompositions for $Q_{n_1+\ldots+n_r}$ by decomposing the cuboids in $\prod_i \mathcal{F}^j_i$ into symmetric chains.
\end{thm}

Using a very specific symmetric chain decomposition provided in \cite{taut} of cuboids formed as the product of at least five $2$-element chains with a chain of length at least $5$, we can remove the goodness hypothesis in Theorem 3.3 completely for the odd dimensional hypercubes provided there are at least six of them appearing in the product. As we will see, cuboids with all but one side of length $2$ prove to be some of the most challenging ones to handle.

\begin{thm}
For $1 \le i \le r$, suppose we have $k \ge 3$ almost orthogonal symmetric chain decompositions $\mathcal{F}^j_i$ of $Q_{n_i}$, $n_i \ge 2$, with $\ge 6$ of the $n_i$ odd, and $\{\mathcal{F}^j_i\}$ is a good collection whenever $n_i$ is even. Then we can construct $k$ almost orthogonal symmetric chain decompositions for $Q_{n_1+\ldots+n_r}$ by decomposing the cuboids in $\prod_i \mathcal{F}^j_i$ into symmetric chains.
\end{thm}

As a consequence, if we can produce $k$ almost orthogonal symmetric chain decompositions of hypercubes whose dimensions have no common factor, such that each is either odd dimensional or good, then we can produce $k$ almost orthogonal symmetric chain decompositions of $Q_n$ for $n$ large enough. We will see such families of decompositions in the next section for $k=3$ and $n=5,7$.

\section{Groundwork}
\subsection{Littlewood-Offord decomposition}
Shearer and Kleitman's paper \cite{SK} uses the Greene-Kleitman method \cite{GK} to create orthogonal chain decompositions, but Kleitman's ``duplication'' technique \cite{LO} used to solve the Littlewood-Offord problem yields a recursive proof more in line with the methods used in this paper, so we will recast their proof using this method (it yields the same decompositions).

Given a symmetric chain decomposition of $Q_n$, we can inductively create a symmetric chain decomposition of $Q_{n+1}$ as follows.

For every chain $A:a_k \subset a_{k+1} \subset \ldots \subset a_{n-k}$ in the decomposition of $Q_n$, let $A'$ be the chain $a_k \cup \{n+1\}\subset \ldots \subset a_{n-k} \cup \{n+1\}$. Neither chain is symmetric in $Q_{n+1}$: $A$ is one element too low, and $A'$ is one element too high. If we move the largest element from $A'$ to $A$ or the smallest element from $A$ to $A'$ however, we now have two symmetric chains (if one of the chains happens to be empty after doing this, discard it). Repeating this for every chain in the decomposition of $Q_n$, we get a decomposition of $Q_{n+1}$ into symmetric chains!

Now, suppose we have two almost orthogonal symmetric chain decompositions $\mathcal{F},\mathcal{G}$ of $Q_n$ (clearly such decompositions exist for $n=2$ say). Then for a chain $A$ in $\mathcal{F}$, we apply the above procedure moving the largest element of $A'$ to $A$, and for a chain $B$ in $\mathcal{G}$, we move the smallest element of $B$ to $B'$. It is an easy check that this yields almost orthogonal symmetric chain decompositions for $Q_{n+1}$. We needed to change how we moved the elements, because if we had for example moved the smallest element of $A$ to $A'$ and the smallest element of $B$ to $B'$, and $A,B$ shared the same smallest element, then $A'$ and $B'$ would have shared the same smallest two elements, contradicting orthogonality.

Hence we have produced a decomposition of $Q_{n+1}$ from $Q_n$, and we are done by induction. To produce orthogonal chain decompositions, it suffices to move the empty set from the maximal chain in $\mathcal{F}$ to one of the one or two element long chains disjoint from the maximal chain in $\mathcal{G}$.

We see immediately what the problem is in propagating larger families of almost orthogonal symmetric chain decompositions to higher dimensional hypercubes---it might be impossible to choose the method of duplicating the chains in such a way that it preserves orthogonality, since we now have three chains meeting at a point instead of two, forbidding us from simply ``doing the opposite'' to each of the families (in fact, it is already impossible if we just consider the three maximal chains).

\subsection{Three almost orthogonal symmetric chain decompositions of $Q_5$}

In order to start creating almost orthogonal symmetric chain decompositions of high dimensional hypercubes, we need decompositions that can be used in a product construction. We do this first for $Q_5$, producing $3$ such decompositions. To compactify notation, we omit the empty set and the whole set from our chains, and we further suppose that if a chain is written, then all chains found by cycling the indices should also be written. An easy check shows that these work. The $2$-element chains form five paths of length $3$, so this family of decompositions is good (see Definition 2.2).\\\\
$\mathcal{F}_1$:
\begin{tabular}{c c c c c}
&1 &12 &123 &1235\\
&  &14 &134 &\\
\hline
\end{tabular}
$\varnothing$ and 12345 added onto the chain starting with 1\\
$\mathcal{F}_2$:
\begin{tabular}{c c c c c}
&1 &15 &135 &1345\\
&  &14 &145 &\\
\hline
\end{tabular}
$\varnothing$ and 12345 added onto the chain starting with 5\\
$\mathcal{F}_3$:
\begin{tabular}{c c c c c}
&1 &14 &124 &1234\\
&  &12 &125 &\\
\hline
\end{tabular}
$\varnothing$ and 12345 added onto the chain starting with 3

\subsection{Three almost orthogonal symmetric chain decompositions of $Q_7$}
We now produce $3$ almost orthogonal symmetric chain decompositions of $Q_7$. Grouping the elements modulo the equivalence relation given by cycling the indices, one can readily check the finite number of cases to show that the decompositions below work. The $2$-element chains form seven $K_{1,3}$ graphs, seven paths of length $2$, and seven paths of length $1$, so this family of decompositions is good (see Definition 2.2).\\\\
$\mathcal{F}_1$:
\begin{tabular}{c c c c c c c}
&1 &12 &123 &1234 &12345 &123456\\
&  &13 &134 &1345 &13456 &\\
&  &14 &147 &1247 &12457 &\\
&  &   &124 &1245 &      &\\
&  &   &135 &1357\\
\hline
\end{tabular}
$\varnothing$ and 1234567 added onto chain starting with 1\\
$\mathcal{F}_2$:
\begin{tabular}{c c c c c c c}
&1 &17 &167 &1567 &14567 &134567\\
&  &13 &137 &1357 &13567 &\\
&  &14 &145 &1345 &13457 &\\
&  &   &134 &1347 &      &\\
&  &   &135 &1235\\
\hline
\end{tabular}
$\varnothing$ and 1234567 added onto chain starting with 7\\
$\mathcal{F}_3$:
\begin{tabular}{c c c c c c c}
&1 &13 &136 &1346 &13467 &123467\\
&  &12 &124 &1247 &12347 &\\
&  &14 &134 &1234 &12346 &\\
&  &   &123 &1236 &      &\\
&  &   &125 &1245        &\\
\hline
\end{tabular}
$\varnothing$ and 1234567 added onto chain starting with 6\\

\subsection{Products of almost orthogonal symmetric chain decompositions}
Note that if $\mathcal{F}_1,\mathcal{F}_2$ are a pair of almost orthogonal symmetric chain decompositions of $Q_m$, and $\mathcal{G}_1,\mathcal{G}_2$ are likewise for $Q_n$, then $\mathcal{F}_1\times \mathcal{G}_1$ and $\mathcal{F}_2\times \mathcal{G}_2$ are two decompositions of $Q_{m+n}$ into products of symmetric chains which have very small intersections between each other.

We will see that the product of 2 hypercubes each with $k$ almost orthogonal symmetric chain decompositions can almost be decomposed into $k$ almost orthogonal symmetric chain decompositions, but a problem arises with chains of small length. This problem will later be mitigated by considering products of $\ge 3$ hypercubes. In Definition 4.1, we introduce notions which will turn out to be indispensable in producing almost orthogonal symmetric chain decompositions with the aid of Lemma 4.2.
\begin{defn}
A chain which \textit{skips no ranks} (i.e. the set of ranks forms an interval in $\mathbb{N}$) inside a product of hypercubes $Q_{n_1}\times\ldots\times Q_{n_r}$ is said to be \textit{very proper} if for every pair of distinct elements $a,a'$ in the chain, there is a coordinate $1 \le i \le r$ where the $i$'th coordinates $a_i$, $a'_i$ differ, and $\{a_i,a'_i\} \ne \{\varnothing, [n_i]\}$. A \textit{proper} chain is defined in the same way except we do not require the condition for $\{a,a'\}$ the smallest and largest element of the product hypercube.
We call a decomposition of a subset of $Q_{n_1}\times\ldots\times Q_{n_r}$ (which does not necessarily lie in a rank symmetric way inside the product) into (very) proper chains all of which are symmetric chains with respect to the subset, a \textit{(very) proper decomposition}.
\end{defn}

Note if a subset of the hypercube product is just a product of chains which skip no ranks in the corresponding hypercubes, and no chain is maximal, then any symmetric chain decomposition is very proper. If at least one chain is not maximal, then all proper decompositions are very proper.

\begin{lem}
Suppose we have $r$ hypercubes $Q_{n_i}$, each with a pair of almost orthogonal symmetric chain decompositions $\mathcal{F}^1_i$ and $\mathcal{F}^2_i$. Take two cuboids $A_1 \times A_2 \times \ldots\times A_r$ and $B_1 \times B_2 \times \ldots\times B_r$, where $A_i \in \mathcal{F}^1_i$, $B_i \in \mathcal{F}^2_i$ are symmetric chains, and suppose the first cuboid has a proper decomposition. Then given any symmetric chain decomposition of the second cuboid, every chain in the first cuboid's decomposition intersects every chain in the second cuboid's  decomposition in at most one element unless all $A_i$ and $B_i$ are maximal, in which case we have the exceptional case that the maximal chains in both intersect in precisely their minimal and maximal elements.
\end{lem}
\begin{proof}
Everything follows straightforwardly from the fact that $A_i \cap B_i$ is either empty, a single element in $Q_{n_i}$, or $\{\varnothing, [n_i]\}$.
\end{proof}

\begin{rmk}
In light of Lemma 4.2, it appears at first glance that we made our definition of (very) proper decompositions unnecessarily general, by not forcing the subset of $Q_{n_1}\times\ldots\times Q_{n_r}$ to lie inside the product in such a way that the notions of symmetric chain for the subset and the product coincide. However, we will see later examples of subsets that arise very naturally as $C_1 \times C_2 \times \ldots \times C_k$, with each $C_i$ a (very) proper chain that skips no ranks in a product of hypercubes, such that the $C_i$ are not necessarily symmetric in their respective products, but the notion of symmetric chain in $C_1 \times C_2 \times \ldots \times C_k$ agrees with the notion in the total ambient product. On our way to creating (very) proper decompositions, when we work with a subset of the $C_i$ we obviously need the more general definition, even though it clearly doesn't affect the final symmetry of the chains we create in the total product.
\end{rmk}
\section{Products of 2 Hypercubes and The Notion of Goodness}
Suppose that we have $k$ almost orthogonal symmetric chain decompositions $\mathcal{F}_i$ of $Q_m$, and $\mathcal{G}_i$ of $Q_n$, with $m,n \ge 2$.

If $A,B$ are chains in $\mathcal{F}_i,\mathcal{F}_j$, and $C,D$ are chains in $\mathcal{G}_i,\mathcal{G}_j$, then $(A\times C) \cap (B \times D)= (A \cap B)\times (C \cap D)$, which has size 0,1,2, or 4. If the size is 0 or 1, then any decomposition of the rectangles $A\times C$ and $B \times D$ into symmetric chains (which one can easily check implies they are symmetric in the ambient $Q_{m+n}$) will have intersection of size $0$ or $1$ between their respective chains, so we only have to worry about size 2 and 4 intersections of the rectangles.

If we manage to give a proper decomposition of a rectangle, then Lemma 4.2 says we don't have to worry about intersections of these chains with any other chains.

Consider the decomposition of an $r \times s$ rectangle depicted in Figure 1 for $r=4$, $s=6$.

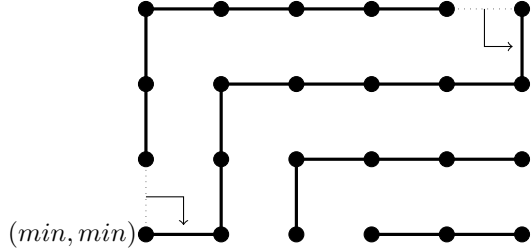
\begin{figure}[H]
\centering
\begin{tikzpicture}
\foreach \x in {0,1,...,5}{                           
    \foreach \y in {0,1,...,3}{                       
    \node[draw,circle,inner sep=2pt,fill] at (\x,\y) {}; 
    }
}
\draw[very thick] (4,3)--(0,3)--(0,1);
\draw[very thick] (5,3)--(5,2)--(1,2)--(1,0)--(0,0);
\draw[very thick] (5,1)--(2,1)--(2,0);
\draw[very thick] (5,0)--(3,0);
\draw[dotted, very thin] (0,0)--(0,1);
\draw[dotted, very thin] (4,3)--(5,3);
\draw (0,0) node [anchor=east] {$(min,min)$};
\draw[->, shorten >=0.125cm] (4.5,3)--(4.5,2.5)--(5,2.5);
\draw[->, shorten >=0.125cm] (0,0.5)--(0.5,0.5)--(0.5,0);
\end{tikzpicture}
\caption{Zigzag decomposition for a $4 \times 6$ rectangle}
\end{figure}

\begin{defn}
For $r,s \ge 3$, we say the \textit{zigzag decomposition} of an $r \times s$ rectangle is the modification of the decomposition by 90 degree clockwise rotated L's as depicted in Figure 1. Specifically, we modify two edges of the two ``leftmost'' chains, at the smallest and largest elements.
\end{defn}

\begin{lem} Given a chain of length $r \ge 3$ in $Q_m$, and a chain of length $s\ge 3$ in $Q_n$, the zigzag decomposition of the product $r \times s$ rectangle is a proper decomposition.
\end{lem}
\begin{proof}
If we draw the zigzag decomposition as in Figure 1, then the salient features which ensure it is a proper decomposition are that no chain horizontally connects a point on the left with the corresponding point on the right, or vertically connects a point on the bottom with the corresponding point on the top.
\end{proof}

Hence, we can properly decompose all rectangles with both sides of length $\ge 3$. Also, if neither chain is maximal, then any decomposition of the rectangle will be (very) proper. We are thus left to consider rectangles that are the product of a maximal chain and a $1$ or $2$-element chain. Now comes a serious problem though: it is impossible to do a proper symmetric chain decomposition of such rectangles! So we have to carefully consider their pairwise intersections.

We can in fact determine exactly the obstruction to decomposing $Q_m \times Q_n$ in this fashion (by decomposing the rectangles in the products $\mathcal{F}_i \times \mathcal{G}_i$). Let $\epsilon_m$ be $1$ if $m$ is even and $2$ if $m$ is odd, and similarly for $\epsilon_n$ (so $\epsilon$ is the length of the smallest chain in the corresponding hypercube). Note that the only rectangles we have left to deal with are size $\epsilon_m \times (n+1)$ and $(m+1) \times \epsilon_n$. Because $\{\mathcal{F}_i\}$ and $\{\mathcal{G}_i\}$ are each almost orthogonal families, an $\epsilon_m \times (n+1)$ and $(m+1) \times \epsilon_n$ rectangle will intersect in at most one point, so we can handle the intersections of $\epsilon_m \times (n+1)$ rectangles and of $(m+1) \times \epsilon_n$ rectangles separately. We consider $\epsilon_m \times (n+1)$ rectangles, $(m+1) \times \epsilon_n$ is identical with $m,n$ swapped.

If $\epsilon_m=1$, then the rectangles are already just chains, and the condition amounts to saying that the decompositions of $Q_m$ have distinct $1$-element chains.

If $\epsilon_m=2$, then there are two ways of decomposing a $2 \times (n+1)$ rectangle into symmetric chains.

\begin{defn}Say a decomposition of a $2 \times r$ rectangle where the longer chain contains the long upper edge of the rectangle the \textit{top decomposition}, and the other one the \textit{bottom decomposition}, as depicted in Figure 2.
\end{defn}
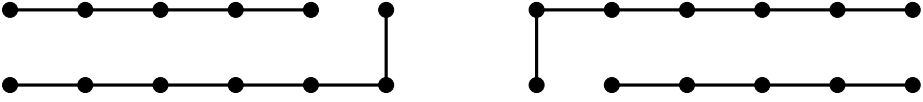
\begin{figure}[H]
\centering
\begin{tikzpicture}
\foreach \x in {0,1,...,5}{                           
    \foreach \y in {0,1}{
    \node[draw,circle,inner sep=2pt,fill] at (\x,\y) {}; 
    }
}

\foreach \x in {7,8,...,12}{                           
    \foreach \y in {0,1}{
    \node[draw,circle,inner sep=2pt,fill] at (\x,\y) {}; 
    }
}
\draw[very thick] (0,0)--(5,0)--(5,1);
\draw[very thick] (0,1)--(4,1);

\draw[very thick] (7,0)--(7,1)--(12,1);
\draw[very thick] (8,0)--(12,0);

\end{tikzpicture}
\caption{Bottom and top decompositions of $2 \times 6$ rectangle respectively}
\end{figure}

If two 2-element chains $A,B$ in $Q_m$ intersect in their top element, then as long as the decompositions of $A \times (n+1)$ and $B \times (n+1)$ are not both top decompositions, the chains will intersect appropriately (and if they were both top decompositions, then the two longest chains would not intersect appropriately). Similarly if $A,B$ intersect in their bottom element, with bottom decompositions. If they don't intersect at all, then it doesn't matter how we decompose the rectangles.

Consider a directed graph whose vertex set is the union of all 2-element chains in all decompositions $\mathcal{F}_i$ of $Q_m$, with edges corresponding to the 2-element chains themselves. Give all edges corresponding to a bottom decomposition the orientation pointing from the smaller subset to the larger one, and if it corresponds to a top decomposition, give it the opposite orientation. Then the corresponding graph-theoretic condition is that the out-degree of every vertex is at most 1.

\begin{defn}
Call a graph \textit{good} if it can be given an orientation such that the out-degree of every vertex is at most 1.
\end{defn}

\begin{lem}
A graph is good if and only if every component has at most one cycle (equivalently, every component is either a tree, or a tree union an edge).
\end{lem}
\begin{proof}
If a component has no cycle, then choose a root vertex and direct all edges towards it. If a component has a unique cycle, give the cycle an orientation, then direct all remaining arrows towards the cycle. If a component has two cycles, then it is easy to see the cycles would have to be disjoint, and then one can check there is no way to orient a path between the two cycles.
\end{proof}

The above discussion shows that the definition of ``goodness'' in Definition 2.2 is precisely what we need to make Theorem 5.6 true.

\begin{thm}
Let $m,n \ge 2$, and suppose $Q_m,Q_n$ each have a family of $k$ almost orthogonal symmetric chain decompositions $\mathcal{F}_i$ and $\mathcal{G}_i$ respectively. Then we can construct $k$ almost orthogonal symmetric chain decompositions of $Q_m \times Q_n=Q_{m+n}$ by decomposing the rectangles in $\mathcal{F}_i\times \mathcal{G}_i$ into symmetric chains if and only if the two families of decompositions are both good.
\end{thm}

We have checked that the decompositions of $Q_5$ and $Q_7$ provided in Sections 4.2 and 4.3 are good when we introduced them. Note that the product of two hypercubes with good families of decompositions may not have a good family of decompositions, so we can't use Theorem 5.6 to produce decompositions of arbitrarily large hypercubes yet. However, we still have the following.

\begin{cor}
$Q_{10}$, $Q_{12}$, $Q_{14}$ all have 3 almost orthogonal symmetric chain decompositions.
\end{cor}

\section{Product of three hypercubes, one odd and two good, or three good}
The goal of this section is to prove the following.

\begin{thm}
Let $m,n,p \ge 2$, and suppose $Q_m$, $Q_n$, $Q_p$ have $k \ge 3$ almost orthogonal symmetric chain decompositions $\mathcal{F}_i$, $\mathcal{G}_i$, and $\mathcal{H}_i$, two of these families are good, and the remaining family is either good or is associated to an odd dimensional hypercube. Then we can construct $k$ almost orthogonal symmetric chain decompositions of $Q_{m+n+p}$ by decomposing the cuboids in $\mathcal{F}_i\times \mathcal{G}_i \times \mathcal{H}_i$ into symmetric chains.
\end{thm}

As we will be working quite a lot with proper and very proper chains, Lemma 6.2 (which we call the Decomposition Lemma) will be frequently used in the rest of the paper. Note that the proof is more subtle than might initially appear, as the chains are in products of hypercubes as opposed to individual hypercubes, so the definition of (very) proper chains must be carefully considered.

\begin{lem}[Decomposition Lemma]
The following statements are true in a product of hypercubes of dimensions $\ge 2$.
\begin{enumerate}
\item The product of two proper chains of lengths $\ge 3$ has a proper decomposition.
\item The product of two very proper chains has a very proper decomposition.
\item The product of a very proper chain of length $\ge 3$ with a proper chain has a very proper decomposition.
\end{enumerate}

In particular we have the following.
\begin{itemize}
\item The product of a set which has a proper decomposition with a (very) proper chain of length $\ge 3$ has a (very) proper decomposition.
\item The product of a set which has a very proper decomposition with a very proper chain has a very proper decomposition.
\end{itemize}
\end{lem}
\begin{proof}
For the second item, we can take any symmetric chain decomposition. The first item follows from the zigzag decomposition, and the third follows from the zigzag decomposition when the proper chain has length $\ge 3$, and from the second item otherwise.
\end{proof}

We will also need the following modified version of the zigzag decomposition.
\begin{defn}
For $r,s \ge 3$, we call a decomposition of an $r \times s$ rectangle into the two modified chains from the zigzag decomposition, one proper of length $r+s-1$, and one very proper of length $r+s-3$, along with the remaining $(r-2) \times (s-2)$ rectangle, which is the product of two very proper chains, the \textit{partial zigzag decomposition}.
\end{defn}

There is a subtlety to Definition 6.3, as neither of the two sides of the remaining rectangle are symmetric, though no problems will arise as was noted in Remark 4.3.

Note that the hypothesis $k \ge 3$ in Theorem 6.1 forces $m,n,p \ge 4$ (see the Introduction). We shall need another assumption --- for emphasis, we display it below.

\begin{equation}
\begin{split}
\text{In our cuboid under consideration, all sides are made up of} 
\\\text{proper chains, and all sides of length }<5\text{ are very proper.}
\end{split}\tag{$\ast$}
\end{equation}

This assumption is true for example when the sides are proper chains inside (products of) hypercubes of dimension $\ge 4$, as is the situation for example in Theorem 6.1.

\begin{lem} A cuboid of size $r_1 \times r_2 \times r_3$ satisfying $(\ast)$ with $r_i \ge 3$ has a proper decomposition.
\end{lem}

\begin{proof}
If we have some $r_i <5$ (say $r_1$), then the corresponding chain is very proper. By Decomposition Lemma 6.2, there is a proper decomposition of the product of the last two factors, and taking the product with the first factor then yields by Decomposition Lemma 6.2 again a (very) proper decomposition.

Suppose then all $r_i \ge 5$. Do a partial zigzag decomposition of $r_1 \times r_2$ into a proper chain of length $r_1+r_2-1$, a very proper chain of length $r_1+r_2-3$, and a rectangle of dimensions $(r_1-2) \times (r_2-2)$ (the product of two very proper chains of length $\ge 3$). Decompose the product of the proper and very proper chains with $r_3$ using Decomposition Lemma 6.2, and for the product of the rectangle with $r_3$, use Decomposition Lemma 6.2 twice.
\end{proof}

Now, let us consider possible cuboids in $\mathcal{F}_i \times \mathcal{G}_i\times \mathcal{H}_i$ that we need to decompose. If none of the three chains is of length 2 or 1, then Lemma 6.4 gives us a proper decomposition of the cuboid. If none of the chains is maximal, then we can decompose the cuboid arbitrarily to get a (very) proper decomposition. Lemmas 6.5, 6.6, and 6.7 properly decompose most of the remaining cuboids.

\begin{lem}
A cuboid of size $2 \times r_2 \times r_3$ satisfying $(\ast)$ with $r_i\ge 4$ has a very proper decomposition.
\end{lem}
\begin{proof}
Take the bottom decomposition of $2 \times r_2$, as depicted in Figure 3.

\begin{figure}[H]
\centering
\begin{tikzpicture}
\foreach \x in {0,1,...,5}{                           
    \foreach \y in {0,1}{
    \node[draw,circle,inner sep=2pt,fill] at (\x,\y) {}; 
    }
}
\draw[very thick] (0,0)--(5,0)--(5,1);
\draw[very thick] (0,1)--(4,1);
\draw (0,0) node [anchor=east] {A};
\draw (5,0) node [anchor=west] {B};
\draw (5,1) node [anchor=west] {C};
\end{tikzpicture}
\caption{Bottom Decomposition of $2 \times r_2$ rectangle with longer chain labeled}
\end{figure}
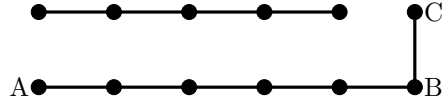

The smaller chain is very proper of length $\ge 3$, so taking the product with $r_3$ is handled by Decomposition Lemma 6.2. For the longer chain, we decompose the product with $r_3$ as in Figure 4. Only 3 chains were modified from a decomposition into clockwise rotated L's, namely the 3 ``leftmost'' chains. Graphically, the salient features that need to be checked in Figure 4 to ensure that we have a very proper decomposition are that the bottom left $A$ is not connected to the upper right $B$, that no chain connects a point on the left side horizontally to the corresponding point on the right side, and that no point on the bottom side connects vertically to the corresponding point on the line segment connecting the two $B$'s.

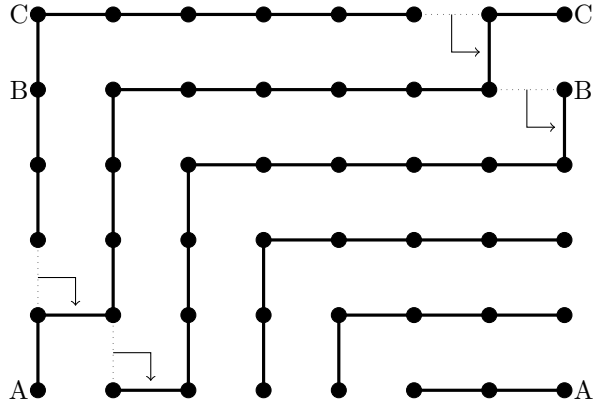
\begin{figure}[H]
\centering
\begin{tikzpicture}
\foreach \x in {0,1,...,7}{                           
    \foreach \y in {0,1,...,5}{                       
    \node[draw,circle,inner sep=2pt,fill] at (\x,\y) {}; 
    }
}
\draw[very thick] (0,2)--(0,5)--(5,5);
\draw[very thick] (0,0)--(0,1)--(1,1)--(1,4)--(6,4)--(6,5)--(7,5);
\draw[very thick] (1,0)--(2,0)--(2,3)--(7,3)--(7,4);
\draw[very thick] (3,0)--(3,2)--(7,2);
\draw[very thick] (4,0)--(4,1)--(7,1);
\draw[very thick] (5,0)--(7,0);
\draw[dotted, very thin] (0,1)--(0,2);
\draw[dotted, very thin] (1,0)--(1,1);
\draw[dotted, very thin] (5,5)--(6,5);
\draw[dotted, very thin] (6,4)--(7,4);
\draw (0,0) node [anchor=east] {A};
\draw (0,4) node [anchor=east] {B};
\draw (7,0) node [anchor=west] {A};
\draw (7,4) node [anchor=west] {B};
\draw (0,5) node [anchor=east] {C};
\draw (7,5) node [anchor=west] {C};
\draw[->, shorten >=0.125cm] (0,1.5)--(0.5,1.5)--(0.5,1);
\draw[->, shorten >=0.125cm] (1,0.5)--(1.5,0.5)--(1.5,0);
\draw[->, shorten >=0.125cm] (5.5,5)--(5.5,4.5)--(6,4.5);
\draw[->, shorten >=0.125cm] (6.5,4)--(6.5,3.5)--(7,3.5);
\end{tikzpicture}
\caption{Decomposition of (longer chain from Figure 3)$\times r_3$}
\end{figure}

\end{proof}

\begin{lem} A cuboid of size $2 \times 3 \times r_3$ satisfying $(\ast)$ with $r_3 \ge 3$ has a very proper decomposition.
\end{lem}
\begin{proof}
Apply Decomposition Lemma 6.2 to the last two chains to get a very proper decomposition, then as the first chain is very proper, apply Decomposition Lemma 6.2 again.
\end{proof}

\begin{lem} A cuboid of size $1 \times r_2 \times r_3$ satisfying $(\ast)$, with $r_2, r_3 \ge 3$ and $r_2$ very proper, has a very proper decomposition.
\end{lem}
\begin{proof}
Apply Decomposition Lemma 6.2 to the last two factors to get a very proper decomposition.
\end{proof}

The only cuboids we have not dealt with are those with one maximal chain and the other two of length $\le 2$, and the rectangles $(m+1) \times (n+1) \times 1$, $(m+1)\times 1 \times (p+1)$, and $1 \times (n+1) \times (p+1)$ if they exist. We note now that the conditions in Theorem 6.1 imply that any factor corresponding to an even dimensional hypercube is good.

Suppose we have an $(m+1) \times (n+1) \times 1$ cuboid. If we do a proper decomposition of the first two factors, then only the maximal chain is not very proper, the properness condition failing precisely at its top and bottom elements, so it could potentially intersect another chain from another cuboid in the two elements $(\varnothing, \varnothing,  \star)$ and $([m], [n], \star)$ (where $\star$ is the element in $Q_p$ which the $1$-element chain corresponds to). Of the aforementioned remaining cuboids we need to worry about, the only one which could possibly intersect in these two elements is another $(m+1) \times (n+1) \times 1$ cuboid, coming from one of the other triples of almost orthogonal symmetric chain decompositions. But because we have a $1$-element chain in the last factor, $Q_p$ must be a good even dimensional hypercube, so this $1$-element chain is different for the two cuboids, so the cuboids are disjoint.

We reason similarly for $(m+1)\times 1 \times (p+1)$ and $1 \times (n+1) \times (p+1)$, so all that is left to consider are the pairwise intersections of cuboids formed as the product of one maximal chain, and two chains of length $\le 2$.

Note that if the maximal chain is in a different factor for two such cuboids, then by the almost orthogonality of the decompositions, they intersect in size at most $1$, so we only have to deal with the intersection properties of those cuboids with the maximal chain in the same factor. This then fixes the dimensions of the cuboids under consideration to be the same.

If one of the factors is length $1$, then that factor corresponds to an even dimensional, hence good, hypercube, so all such cuboids are disjoint and there's nothing to check. Otherwise, such a cuboid has at least one of its two length $2$ factors in a good hypercube, say $Q_n$ (we only use goodness in this factor now). By symmetry, the only case we have to deal with are cuboids of size $2 \times 2 \times (p+1)$.

For these cuboids, ignore the $Q_m$ factor and decompose the last two factors using top and bottom decompositions like in Section 3 (using the goodness of $Q_n$), so that we're left with rectangles of the form $2 \times C$, with $C$'s corresponding to different $\mathcal{G}_i \times \mathcal{H}_i$ having intersections of size at most $1$. Then any decompositions of the $2 \times C$ rectangles work because $\mathcal{F}_i$ is an almost orthogonal family.

\section{Product of four odd hypercubes, two good}
The goal of this section is to prove Theorem 7.1.
\begin{thm}Take $m,n,p,q \ge 3$ odd. Suppose $Q_m$, $Q_n$, $Q_p$, and $Q_q$ have families of $k \ge 3$ almost orthogonal symmetric chain decompositions $\mathcal{F}_i$, $\mathcal{G}_i$, $\mathcal{H}_i$, and $\mathcal{K}_i$, and two of these families are good. Then we can construct $k$ almost orthogonal symmetric chain decompositions of $Q_{m+n+p+q}$ by decomposing the cuboids in $\mathcal{F}_i\times \mathcal{G}_i \times \mathcal{H}_i \times \mathcal{K}_i$ into symmetric chains.
\end{thm}

Note that as before, $k \ge 3$ implies $m,n,p,q \ge 4$, so $(\ast)$ holds for all cuboids we consider. Lemmas 7.2, 7.3, 7.4, 7.5, and 7.6 allow us to decompose most of the cuboids in the product in a proper way.

\begin{lem}
A cuboid of size $r \times s \times t \times u$ satisfying $(\ast)$, with $r \ge 3$ very proper, $s,t,u \ge 2$, has a very proper decomposition. 
\end{lem}
\begin{proof}
If the $s \times t \times u$ cuboid has a proper decomposition, then we're done by Decomposition Lemma 6.2. From Lemmas 6.4, 6.5, and 6.6, the only case we have left is when (up to permutation), $s=t=2$ and $u \ge 2$. But then the $r\times u$ cuboid has a very proper decomposition by Decomposition Lemma 6.2, so since the other two factors are very proper, we're done by applying Decomposition Lemma 6.2 two more times.

\end{proof}

\begin{lem}
A cuboid of size $2 \times s \times t \times u$ satisfying $(\ast)$, $s,t,u \ge 5$ has a very proper decomposition.
\end{lem}
\begin{proof}
Take $t \times u$ and do a partial zigzag decomposition, to get a proper chain of length $t+u-1$, a very proper chain of length $t+u-3$, and a rectangle of dimension $(t-2) \times (u-2)$, (the product of two very proper chains of length $\ge 3$). For the chains, take the product with $2 \times s$ and use one of Lemmas 6.5 or 6.6. For the rectangle, use Decomposition Lemma 6.2 on $s \times (t-2)$ to get a very proper decomposition, then since the other two factors are very proper, we're done by applying Decomposition Lemma 6.2 two more times.

\end{proof}

\begin{lem}
A cuboid of size $2 \times 2 \times t \times u$ satisfying $(\ast)$, $t,u \ge 5$ has a very proper decomposition.
\end{lem}
\begin{proof}
By Lemma 6.5, $2 \times t \times u$ has a very proper decomposition, and since the first $2$ is very proper, we're done by Decomposition Lemma 6.2.
\end{proof}

\begin{lem}
A cuboid of size $2 \times 2 \times 2 \times 2$ satisfying $(\ast)$ has a very proper decomposition.
\end{lem}
\begin{proof}
All of the chains are very proper, so we apply Decomposition Lemma 6.2 repeatedly to get our result.
\end{proof}

\begin{lem}
A cuboid of size $r \times s \times t \times u$ satisfying $(\ast)$, $r,s,t,u \ge 5$ has a very proper decomposition.
\end{lem}
\begin{proof}

Take $r\times s$, and do a partial zigzag decomposition, to get a proper chain of length $r+s-1$, a very proper chain of length $r+s-3$, and a rectangle of dimension $(r-2)\times (s-2)$ (the product of two very proper chains of length $\ge 3$). For the chains, take the product with $t\times u$ and use Lemma 6.4. For the rectangle, use Decomposition Lemma 6.2 on $t \times u$ to get a proper decomposition, then since $r-2$ and $s-2$ are very proper of length $\ge 3$, we're done by applying Decomposition Lemma 6.2 two more times.

\end{proof}

The only cuboids we haven't created proper decompositions for are ones with three factors of length 2, and one factor maximal. Two such cuboids can intersect in more than one element only if the maximal chain is in the same factor. We will show how to handle $2 \times 2 \times 2 \times (q+1)$ cuboids, the other cases are similarly dealt with. As two of the decompositions are good, without loss of generality assume $Q_p$ (the third factor) is good.

For $2 \times 2 \times 2 \times (q+1)$, ignore the first two factors and decompose the last two factors using top and bottom decompositions like in Section 3 (using the goodness of $Q_p$), so that we're left with cuboids of the form $2 \times 2 \times C$, with $C$'s corresponding to different $\mathcal{H}_i \times \mathcal{K}_i$ having intersections of size at most $1$. Then any decompositions of these $2 \times 2 \times C$ cuboids will work because $\mathcal{F}_i$ and $\mathcal{G}_i$ are almost orthogonal families.

\section{Arbitrarily many good even hypercubes}
We now consider the product of arbitrarily many good even hypercubes.
\begin{thm}
For $1 \le i \le r$, suppose we have $k \ge 3$ good almost orthogonal symmetric chain decompositions $\mathcal{F}^j_i$ of $Q_{n_i}$, all $n_i \ge 2$ even. Then we can construct $k$ almost orthogonal symmetric chain decompositions for $Q_{n_1+\ldots+n_r}$ by decomposing the cuboids in $\prod_i \mathcal{F}^j_i$ into symmetric chains.
\end{thm}

Note that as before, $k \ge 3$ implies $n_i \ge 4$, so $(\ast)$ holds for all cuboids we consider. We start by generalizing Lemmas 6.4 and 7.6.

\begin{lem}
A cuboid of size $r_1 \times r_2 \times \ldots \times r_t$ satisfying $(\ast)$, $r_i \ge 5$ has a proper decomposition.
\end{lem}
\begin{proof}
We've already seen this for $t=2,3,4$ in Decomposition Lemma 6.2, Lemma 6.4, and Lemma 7.6 respectively. We proceed by induction --- assume we always have a proper decomposition for such cuboids of dimension $<t$. Take the last two factors $r_{t-1} \times r_t$, and do a partial zigzag decomposition to get a proper chain of length $r_{t-1}+r_t-1$, a very proper chain of length $r_{t-1}+r_t-3$, and a rectangle of dimensions $(r_{t-1}-2)\times (r_t-2)$ (the product of two very proper chains of length $\ge 3$). For the chains, use the inductive hypothesis for $k-1$. For the rectangle, take a proper decomposition of the remaining $k-2$ factors (which exists by induction), then since $(r_{t-1}-2)$ and $(r_t-2)$ are very proper of length $\ge 3$, we're done by applying Decomposition Lemma 6.2 two more times.
\end{proof}

\begin{lem}
A cuboid of size $r_1 \times r_2 \times \ldots \times r_t$ satisfying $(\ast)$, $r_1 \ge 3$ very proper has a very proper decomposition.
\end{lem}
\begin{proof}
Do a proper decomposition of all factors of length $\ge 5$ by Lemma 8.2, then apply Decomposition Lemma 6.2 with $r_1$ to get a very proper decomposition. The remaining factors are all very proper, so we're done by repeated applications of Decomposition Lemma 6.2.
\end{proof}

The only cuboids we have not dealt with are the ones which are the products of chains of maximal size and minimal size (size 1), and contain at least one chain of minimal size.

For each of these, do a proper decomposition on the factors containing the chains of maximal size. Every chain in such a decomposition is very proper except the maximal chain in the cuboid, and the only two elements the maximal chain could have in common with any other chain are its maximal and minimal elements as this is the only pair of elements that properness fails at. Hence, if two chains intersected in two elements, it would be two maximal chains intersecting in their common top and bottom elements, forcing the cuboids to have the same dimensions. But then they share a factor with a $1$-element chain, which by the goodness of the factor containing the $1$-element chain forces the cuboids to be disjoint.

\section{4 hypercubes, one odd and three good even}
We now consider the last case needed to prove Theorems 3.3 and 3.4.
\begin{thm}
Take $m \ge 3$ odd, and $n,p,q \ge 2$ even. Suppose $Q_m$, $Q_n$, $Q_p$, and $Q_q$ have families of $k \ge 3$ almost orthogonal symmetric chain decompositions $\mathcal{F}_i, \mathcal{G}_i$, $\mathcal{H}_i$, and $\mathcal{K}_i$, with the decompositions of $Q_n$, $Q_p$, and $Q_q$ good. Then we can construct $k$ almost orthogonal symmetric chain decompositions of $Q_{m+n+p+q}$ by decomposing the cuboids in $\mathcal{F}_i\times \mathcal{G}_i \times \mathcal{H}_i \times \mathcal{K}_i$ into symmetric chains.
\end{thm}

Note that as before, $k \ge 3$ implies $m,n,p,q \ge 4$, so $(\ast)$ holds for all cuboids we consider.

\begin{proof}
First note that Lemma 8.3 reduces us to considering cuboids formed by products of maximal and minimal length chains.

Consider first when the factor in $Q_m$ is a $2$-element chain. If none of the factors is maximal, then we are done by Decomposition Lemma 6.2. Lemma 7.3 shows we have a very proper decomposition if the remaining factors are all maximal, and Lemma 6.5 shows we have a very proper decomposition when two of the remaining factors are maximal and one is minimal (of length $1$) by Decomposition Lemma 6.2. Hence the remaining cuboids we have to deal with have sizes $2 \times (n+1) \times 1 \times 1$, $2 \times 1 \times (p+1) \times 1$, or $2 \times 1 \times 1 \times (q+1)$. Any two such cuboids have a $1$-element chain in the same factor, which by goodness of the corresponding hypercube forces them to be disjoint. We will shortly return to consider their intersections with other cuboids.

Assume now that the factor in $Q_m$ is a maximal chain. Lemma 8.2 gives us a proper decomposition when all of the factors are maximal. When there is at least 1 minimal chain in this case, then just as in the end of the proof of Theorem 8.1, any two such cuboids will have no problems with their pairwise intersections when we properly decompose the non-$1$ factors.

Thus we only have to consider the intersections of the remaining cuboids in the first case with cuboids in the second case with at least one factor of length $1$. By almost orthogonality, the maximal length chain in the cuboid from the first case must be in the same factor as a maximal length chain in the cuboid from the second case for the cuboids to have intersection of size at least 2. But then the cuboids must share a factor with a $1$-element chain, forcing disjointness by the goodness of that factor.
\end{proof}

\section{Proof of Theorems 3.1 and 3.3}
We will prove Theorem 3.3 through a series of lemmas, and then prove Theorem 3.1 as a consequence. Note that as before, $k \ge 3$ implies $n_i \ge 4$, so $(\ast)$ holds for all cuboids we consider. First note that Theorem 8.1 handles the case when all hypercubes are even dimensional, so we may assume at least one hypercube is odd dimensional.

\begin{lem}
Theorem 3.3 is true when exactly one of the $r$ hypercubes is odd dimensional.
\end{lem}
\begin{proof}
Applying Theorem 5.6 once if necessary, we reduce to the product of an even number of good even dimensional hypercubes, and one (possibly not good) odd dimensional hypercube. We can now repeatedly apply Theorem 6.1 to the odd hypercube and two even ones until there is only one hypercube left.
\end{proof}
\begin{lem}
Theorem 3.3 is true when exactly two of the $r$ hypercubes are odd dimensional.
\end{lem}
\begin{proof}
If there is at most one even dimensional hypercube, we're done by Theorem 5.6 or Theorem 6.1. Otherwise, apply either Theorem 5.6 or Theorem 6.1 with exactly one odd dimensional hypercube so that we're now left with an odd number of even dimensional hypercubes, and two odd dimensional hypercubes, one of which is possibly not good. Take the possibly not good odd dimensional hypercube and repeatedly apply Theorem 6.1 with two even dimensional hypercubes to reduce down to one even dimensional hypercube, and two odd dimensional hypercubes, one of which is possibly not good. Now apply Theorem 6.1 to the remaining 3 hypercubes.
\end{proof}
\begin{lem}
Assume Theorem 3.3 is true when there are at least three odd dimensional hypercubes and at most 1 even dimensional hypercube. Then Theorem 3.3 is true when there are at least 3 odd dimensional hypercubes.
\end{lem}
\begin{proof}
Suppose we have at least $2$ even dimensional hypercubes. We have at least 2 good odd dimensional hypercubes, take a third one. Repeatedly applying Theorem 6.1 and 9.1 with this third hypercube and some of the even dimensional hypercubes, we reduce to the case of the same number of odd dimensional hypercubes, but with at most 1 even dimensional hypercube.
\end{proof}
\begin{lem}
Theorem 3.3 is true when there are no even dimensional hypercubes, and $r \ge 3$.
\end{lem}
\begin{proof}
Using Lemma 8.3, all product cuboids have a very proper decomposition except possibly for the ones formed by products of maximal chains, and $2$-element chains. If all chains are maximal, then we have a proper decomposition by Lemma 8.2, so assume we have at least one $2$-element chain. If all chains have length $2$, then the cuboid is a product of very proper chains, so we're done by Decomposition Lemma 8.2. If there are exactly $2$ maximal chains, then using Lemma 6.5 on the $2$ maximal chains and a $2$-element chain, we are done by Decomposition Lemma 6.2. If there are at least $3$ maximal chains, do a proper decomposition of all but two of the maximal chains by Lemma 8.2, which reduces us to two possible types of cuboids. Either we get a product of a very proper chain with the two maximal chain case above, hence we are again done by Decomposition Lemma 6.2, or we get a product of 3 proper chains of length $\ge 5$ with some non-zero number of $2$-element chains. In the latter case, we apply Lemma 7.3 with one of the $2$-element chains and the $3$ proper chains, and then conclude by Decomposition Lemma 6.2 for each of the resulting cuboids, which are all products of very proper chains.

Hence all that remains is the case when all but one factor is a $2$-element chain, and the remaining factor is maximal. Two such cuboids can intersect in more than one element only if they have the maximal chain in the same factor. We can restrict ourselves to assuming without loss of generality that the last factor has the maximal chain, so the dimensions are $2 \times 2 \times \ldots \times 2 \times (n_r+1)$. Then without loss of generality, the second last hypercube is good, so ignoring all but the last two factors, we can decompose these factors using top and bottom decompositions like in Section 3 (using the goodness of $Q_{n_{r-1}}$). Then we're left with cuboids of the form $2 \times 2 \times \ldots \times C$, and two such cuboids arising from different $\prod_i \mathcal{F}^j_i$ intersect in at most one element by the almost orthogonality of the remaining factors. Thus arbitrarily decomposing the resulting cuboids gives us the desired decomposition.
\end{proof}
\begin{lem}
Theorem 3.3 is true when there is one even dimensional hypercube, and at least three odd dimensional hypercubes.
\end{lem}
\begin{proof}
Using Lemma 8.3, all product cuboids have a very proper decomposition except possibly for the ones formed by products of maximal chains, and minimal length chains.

Consider first the case when the even dimensional hypercube (which we can assume to be in the last factor) has a maximal length chain in it. Then we can argue exactly as in the proof of Lemma 10.4 to get proper decompositions of all such cuboids except ones of the form $2 \times 2 \times \ldots \times (n_r+1)$.

Next, consider the case that the even dimensional hypercube factor has a $1$-element chain. Then as the even dimensional hypercube is good, any two such cuboids will be disjoint. Furthermore, they will intersect any $2 \times 2 \times \ldots \times (n_r+1)$ cuboid in at most one element. Consequently, we can reduce ourselves to only considering $2 \times 2 \times \ldots \times 2 \times (n_r+1)$ cuboids, and the proof then finishes off identically to the proof of Lemma 10.4.
\end{proof}

This completes the proof of Theorem 3.3. Theorem 3.1, and hence also Corollary 3.2, now follows from Corollary 10.6, by using the $3$ almost orthogonal symmetric chain decompositions of $Q_5$ and $Q_7$ in Subsections 4.2 and 4.3 respectively.

\begin{cor}
We have 3 almost orthogonal symmetric chain decompositions for $Q_n$ for $n \ge 5$ with the possible exceptions of $6,8,9,11,13,16,18,23$ (which are the numbers not non-negative linear combinations of 5 and 7).
\end{cor}
\begin{cor}
We have 3 orthogonal decompositions for $Q_n$ for $n \ge 4$ with the possible exceptions of $9,11,13,23$.
\end{cor}
\begin{proof}
Given a collection of orthogonal decompositions of $Q_{2k-1}$, we can construct a collection of orthogonal decomposition for $Q_{2k}$ by, for each decomposition, duplicating each chain and adding the element $2k$ into each element of the duplicate chain. We leave $n=4$ to the reader (or see \cite{SK}).
\end{proof}

\section{Proof of Theorem 3.4}
In this section, we will show how to decompose a product of hypercubes if there are at least 6 odd dimensional hypercubes, and all the even dimensional hypercubes are good using \cite{taut}. Note that as before, $k \ge 3$ implies $n_i \ge 4$, so $(\ast)$ holds for all cuboids we consider. By repeatedly applying Theorem 6.1 or 9.1 with one of the odd dimensional hypercubes, we can reduce to the case of at most one even dimensional hypercube.

The proofs of Lemmas 10.4 and 10.5 now apply verbatim up until we are only left with $2 \times 2 \times \ldots \times 2 \times (n_r+1)$ cuboids (with at least five 2's). No clever manipulations allow one to (very) properly decompose such cuboids from the results we have, but as it turns out, one can bootstrap a certain decomposition of $2 \times 2 \times 2 \times 2 \times 2 \times 5$ to decompose the product of at least five $2$'s with a chain of length at least $5$ in a very proper way. We refer to \cite{taut} for the very explicit verification that such decompositions exist.

\begin{ack*}
I would like to thank Mitchell Lee for helpful discussions.
\end{ack*}

\end{document}